\documentclass[11pt,a4paper]{article}
\usepackage[utf8]{inputenc}
\usepackage{amsmath}
\usepackage{amsthm}
\usepackage{amsfonts}
\usepackage{amssymb}
\author{Mark Shusterman}
\title{Ascending chains of finitely generated subgroups}
\usepackage{mathtools}
\usepackage{graphicx}
\usepackage{hyperref}
\newtheorem{theorem}{Theorem}[section]
\newtheorem{lemma}[theorem]{Lemma}
\newtheorem{proposition}[theorem]{Proposition}
\newtheorem{corollary}[theorem]{Corollary}
\numberwithin{equation}{section}

\newcommand{\lemref}[1]{\hyperref[#1]{Lemma \ref*{#1}}}
\newcommand{\thmref}[1]{\hyperref[#1]{Theorem \ref*{#1}}}
\newcommand{\propref}[1]{\hyperref[#1]{Proposition \ref*{#1}}}
\newcommand{\corref}[1]{\hyperref[#1]{Corollary \ref*{#1}}}

\makeatletter
\def\moverlay{\mathpalette\mov@rlay}
\def\mov@rlay#1#2{\leavevmode\vtop{%
   \baselineskip\z@skip \lineskiplimit-\maxdimen
   \ialign{\hfil$\m@th#1##$\hfil\cr#2\crcr}}}
\newcommand{\charfusion}[3][\mathord]{
    #1{\ifx#1\mathop\vphantom{#2}\fi
        \mathpalette\mov@rlay{#2\cr#3}
      }
    \ifx#1\mathop\expandafter\displaylimits\fi}
\makeatother

\makeatletter
\newcommand*{\defeq}{\mathrel{\rlap{%
                     \raisebox{0.27ex}{$\m@th\cdot$}}%
                     \raisebox{-0.27ex}{$\m@th\cdot$}}%
                     =}
\makeatother

\begin{document}

\maketitle

\abstract{ We show that a nonempty family of $n$-generated subgroups of a pro-$p$ group has a maximal element. This suggests that 'Noetherian Induction' can be used to discover new features of finitely generated subgroups of pro-$p$ groups. To demonstrate this, we show that in various pro-$p$ groups $\Gamma$ (e.g. free pro-$p$ groups, nonsolvable Demushkin groups) the commensurator of a finitely generated subgroup $H \neq 1$ is the greatest subgroup of $\Gamma$ containing $H$ as an open subgroup. We also show that an ascending sequence of $n$-generated subgroups of a limit group must terminate (this extends the analogous result for free groups proved by Takahasi, Higman, and Kapovich-Myasnikov).}

\section{Introduction}

Chain conditions play a prominent role in Algebra. A good example is the variety of results on Noetherian rings and their modules. In this work we consider chain conditions on profinite groups. All the group-theoretic notions considered for these groups should be understood in the topological sense, i.e. subgroups are closed, homomorphisms are continuous, generators are topological, etc. The ascending chain condition on finitely generated subgroups does not hold for pro-$p$ groups in general, and our first result is some kind of remedy for this.

\begin{theorem} \label{FirstRes}

Let $p$ be a prime number, let $n \in \mathbb{N}$, let $\Gamma$ be a pro-$p$ group, and let $\mathcal{F} \neq \emptyset$ be a family of $n$-generated subgroups of $\Gamma$. Then $\mathcal{F}$ has a maximal element with respect to inclusion.

\end{theorem}

It is our hope that this simple result will unveil new properties of pro-$p$ groups and their finitely generated subgroups. An example is the following theorem, for which we need a definition. Given a prime number $p$, we say that a pro-$p$ group $\Gamma$ has a \textit{\textbf{Hereditarily Linearly Increasing Rank}} if for every finitely generated subgroup $H \leq_c \Gamma$ there exists an $\epsilon > 0$ such that for any open subgroup $U \leq_o H$ we have 

\begin{equation} \label{DefLIREq}
d(U) \geq \max \{d(H), \epsilon(d(H) - 1)[H : U]\}
\end{equation}
where $d(K)$ stands for the smallest cardinality of a generating set for the pro-$p$ group $K$. That is, our definition says that the minimal number of generators of finite index subgroups of $H$ grows monotonically, and linearly as a function of the index. Examples of groups with this property include free pro-$p$ groups, nonsolvable Demushkin groups, and groups from the class $\mathcal{L}$ all of whose abelian subgroups are procyclic.\footnote{By \cite[Theorem 4.6]{SZ0} the deficiency is $\geq 2$, and by the paragraph following \cite[Proposition 1]{HS} it grows linearly for open subgroups, so the number of generators grows linearly as well, and is monotonic by \cite[Theorem B]{KZ}.}

\begin{theorem} \label{SecondRes}

Let $p$ be a prime number, let $\Gamma$ be a pro-$p$ group with a hereditarily linearly increasing rank, and let $H \leq_c \Gamma$ be a nontrivial finitely generated subgroup. Then the commensurator of $H$ in $\Gamma$ $$
\mathrm{Comm}_{\Gamma}(H) \defeq \{g \in \Gamma \ | \ [H : H \cap gHg^{-1}],[gHg^{-1} : H \cap gHg^{-1}] < \infty \}$$
is the greatest subgroup of $\Gamma$ containing $H$ as an open subgroup. Also, the action of any $\mathrm{Comm}_{\Gamma}(H) \lneq_c L \leq_c \Gamma$ on $L/H$ is faithful.

\end{theorem}

Note that several assertions are encompassed in this statement. First, we claim that the family of subgroups of $\Gamma$ containing $H$ as an open subgroup has a (necessarily unique) greatest element with respect to inclusion. In particular,  there are only finitely many subgroups of $\Gamma$ that contain $H$ as an open subgroup. The second assertion is that the aforementioned greatest element is the commensurator of $H$ in $\Gamma$. Hence $H$ is an open subgroup of $\mathrm{Comm}_{\Gamma}(H)$, and thus also an open subgroup of its normalizer in $\Gamma$. Furthermore, given finitely generated commensurable $H_1,H_2 \leq_c \Gamma$ (i.e. $H_1 \cap H_2$ is open in both $H_1$ and $H_2$) we can apply \thmref{SecondRes} to $H_1 \cap H_2$ and conclude that it is an open subgroup of $\langle H_1 \cup H_2 \rangle$. At last, if $[\Gamma : H]$ is infinite, then by taking $L = \Gamma$ we find that $H$ contains no nontrivial normal subgroup of $\Gamma$.

Results analogous to these assertions are abundant in the literature, where the group $\Gamma$ is replaced by: 
\begin{enumerate}

\item Free groups - \cite[Corollary 8.8, Proposition 8.9]{KM}, \cite[Theorem 1]{KS}, \cite{R}.

\item Fuchsian groups - \cite{G}.

\item Hyperbolic groups - \cite[Theorems 1,3]{KaS}.

\item Limit groups - \cite[Theorem 1]{BH}, \cite[Theorem 4.1]{BHMS}, \cite[Chapters 4,5]{NS}, \cite[Section 6]{SZ0}.

\item Groups with a positive first $\ell^2$-Betti number - \cite[Corollary 5.13, Proposition 7.3]{PT}.

\item Free profinite groups - \cite[Main Theorem]{J}.

\item Absolute Galois groups of Hilbertian fields - \cite{Har}.

\item Free pro-$p$ groups and free pro-$p$ products - \cite[3.3, 3.5]{L}, \cite[Theorem C]{KZ}.

\item Nonsolvable pro-$p$ Demushkin groups and other pro-$p$ $IF$-groups with positive deficiency - \cite[3.12, 3.13]{K}, \cite{SZ}, \cite{HS}.

\item Pro-$p$ groups from the class $\mathcal{L}$ - \cite[Theorem C (5-7)]{SZ0}. 

\end{enumerate} 

It is our point of view that an assumption on the increase in the number of generators upon passing to finite index subgroups (e.g.  \eqref{DefLIREq}) creates a good framework for proving results like those stated in \thmref{SecondRes}, the paragraph following it, and the list above. Indeed, all the groups in the list (excluding some of those in 3), exhibit a linear growth of the number of generators as a function of the index of a subgroup (i.e. these groups have rank gradient $>0$). As a result, arguments from the proof of \thmref{SecondRes} can be used to obtain most of the results in the list above. For instance, \cite[pro-$p$ Greenberg Theorem]{L}, \cite[Lemma 3.12, Proposition 3.13]{K}, \cite[Theorem A]{SZ}, and a part of \cite[Theorem C (5-7)]{SZ0} are special cases of \thmref{SecondRes}. 

Next, we generalize Takahasi's theorem (see \cite[Theorem 1]{Tak}, \cite[Lemma]{Hig}, and \cite[Theorem 14.1]{KM}) which is the case of free $G$ in the following.

\begin{theorem} \label{ThirdRes}

Let $G$ be a group for which every subgroup $H \leq G$ whose profinite completion is finitely generated, is itself finitely generated. Let $n \in \mathbb{N}$, and let $\mathcal{F} \neq \emptyset$ be a family of $n$-generated subgroups of $G$. Then $\mathcal{F}$ contains a maximal element.

\end{theorem}

Most notably, by \propref{ProfCompLimProp}, the theorem applies to limit groups, and also to Fuchsian groups. A consequence of \thmref{ThirdRes} is \corref{AutCor}.

\section{Profinite Groups}

\subsection{Directed families}

Given a set $I$, we say that a family of subgroups $\{A_i\}_{i \in I}$ of a group $G$ is directed if for every $i,j \in I$ there exists a $k \in I$ such that $A_k \geq A_i, A_j$. In this case, the abstract subgroup generated by the $\{A_i\}_{i \in I}$ is just their union. Furthermore, it follows by induction that for all $m \in \mathbb{N}$ and $i_1, \dots, i_m \in I$ there exists some $i \in I$ such that $A_i \geq A_{i_k}$ for each $1 \leq k \leq m$.

\begin{lemma} \label{BasicProfLem}

Let $\Gamma$ be a profinite group, let $\{A_i\}_{i \in I}$ be a directed family of subgroups of $\Gamma$, set $A \defeq \langle A_i \rangle_{i \in I}$, let $G$ be a finite group, and let $\tau \colon A \to G$ be an epimorphism. Then there exists some $j \in I$ such that $\tau|_{A_j}$ is a surjection.

\end{lemma}

\begin{proof}

Note that
\begin{equation} \label{GroupCoverEq}
G = \tau(A) = \tau(\langle A_i \rangle_{i \in I}) = \tau(\overline{\bigcup_{i \in I} A_i}) \subseteq \overline{\tau(\bigcup_{i \in I} A_i)} = \bigcup_{i \in I} \tau(A_i)
\end{equation}
so for each $g \in G$ there exists some $i_g \in I$ such that $g \in \tau(A_{i_g})$. Since $G$ is finite, directedness implies that there exists some $j \in I$ such that $A_j \geq A_{i_g}$ for all $g \in G$. It follows that $\tau(A_j) = G$ as required.

\end{proof}

\begin{corollary} \label{BasicProfCor}

Let $\Gamma$ be a profinite group, let $n \in \mathbb{N}$, and let $\{A_i\}_{i \in I}$ be a directed family of $n$-generated subgroups of $\Gamma$. Then $A \defeq \langle A_i \rangle_{i \in I}$ is $n$-generated.

\end{corollary}

\begin{proof}

Let $\tau \colon A \to G$ be an epimorphism onto a finite group. By \lemref{BasicProfLem}, there exists some $j \in I$ such that $\tau(A_j) = G$. Hence, $d(G) \leq d(A_j) \leq n$, so $d(A) \leq n$ since by \cite[Lemma 2.5.3]{RZ} we know that $d(A)$ is determined by the finite homomorphic images of $A$.

\end{proof}

For the proof of \thmref{FirstRes} recall that the Frattini subgroup of a profinite group $U$, denoted by $\Phi(U)$, is defined to be the intersection of all maximal subgroups of $U$.

\begin{theorem} \label{FratThm}

Let $p$ be a prime number, let $n \in \mathbb{N}$, let $\Gamma$ be a pro-$p$ group, and let $\mathcal{F} \neq \emptyset$ be a family of $n$-generated subgroups of $\Gamma$. Then $\mathcal{F}$ contains a maximal element.

\end{theorem}

\begin{proof}

Let $\mathcal{C}$ be a chain in $\mathcal{F}$, and let $U$ be the subgroup of $\Gamma$ generated by the subgroups in $\mathcal{C}$. Since $\mathcal{C}$ is a chain, it is directed, so  $d(U) \leq n$ by \corref{BasicProfCor}. As $U$ is a finitely generated pro-$p$ group, \cite[Proposition 2.8.10]{RZ} tells us that $U \to U/\Phi(U)$ is an epimorphism onto a finite group. By \lemref{BasicProfLem}, there exists some $H \in \mathcal{C}$ such that $H\Phi(U) =  U$, so in view of \cite[Corollary 2.8.5]{RZ} we must have $H = U$. Thus $U \in \mathcal{F}$ is an upper bound for $\mathcal{C}$. By Zorn's Lemma, $\mathcal{F}$ has a maximal element.

\end{proof}

Note that our assumption that $\Gamma$ is not merely a profinite group but a pro-$p$ group, has been used in the proof to conclude that $\Phi(U) \leq_o U$ for any finitely generated $U \leq_c \Gamma$. By \cite[Proposition 2.8.11]{RZ}, this conclusion holds under the weaker assumption that $\Gamma$ is a prosupersolvable group with order divisible by only finitely many primes. Such groups have been studied, for instance, in \cite{AS}.

\subsection{Hereditarily linearly increasing rank}

\subsubsection{Basic properties}

\begin{proposition} \label{UpWProp}

Let $p$ be a prime number, let $\Gamma$ be a pro-$p$ group with a hereditarily linearly increasing rank, let $H \leq_c \Gamma$ be a finitely generated subgroup, and let $R \leq_c \Gamma$ be a subgroup containing $H$ as an open subgroup. Then $d(R) \leq d(H)$.

\end{proposition}

\begin{proof}

By taking the union of a finite generating set for $H$ with a transversal for $H$ in $R$ we get a finite generating set for $R$. It follows from \eqref{DefLIREq} that $d(H) \geq d(R)$ as required.

\end{proof}

\begin{proposition} \label{TorFreeProp}

Let $p$ be a prime number, and let $\Gamma$ be a pro-$p$ group with a hereditarily linearly increasing rank. Then $\Gamma$ is torsion-free.

\end{proposition}

\begin{proof}

Let $C \leq_c \Gamma$ be a finite subgroup. Since $\{1\} \leq_o C$, \eqref{DefLIREq} implies that $0 = d(\{ 1 \}) \geq d(C)$ which guarantees that $C = \{1\}$ as required. 

\end{proof}

\begin{corollary} \label{FreqFamCor}

Let $p$ be a prime number, let $\Gamma$ be a pro-$p$ group with a hereditarily linearly increasing rank, and let $H$ be a finitely generated subgroup of $\Gamma$. Then $\mathcal{F} \defeq \{R \leq_c \Gamma \ | \ H \leq_o R\}$ has a maximal element.

\end{corollary}

\begin{proof}

This is immediate from \propref{UpWProp} and \thmref{FratThm}.

\end{proof}

The following simple lemma is useful for constructing small generating sets of profinite groups. It remains true both for abstract groups and for profinite groups which are not finitely generated, but we do not need it in this generality. 

\begin{lemma} \label{GenCountLem}

Let $G$ be a finitely generated profinite group, let $K \lhd_c G$ be a normal subgroup, and let $H \leq_c G$ be a finitely generated subgroup containing $K$. Then $d(G) \leq d(H) + d(G/K).$

\end{lemma}

\begin{proof}

Let $\pi \colon G \to G/K$ be the quotient map, let $X$ be a finite generating set of $H$, and let $Y$ be a finite generating set of $G/K$. Since $\pi$ is surjective, there exists some $Z \subseteq G$ which is bijectively mapped by $\pi$ onto $Y$. Set $S \defeq X \cup Z, \ M \defeq \langle S \rangle$ and note that it suffices to show that $M = G$, since this guarantees the existence of a generating set (i.e. $S$) for $G$ of the required cardinality.

In order to see that $M=G$, note that $\pi(M) \supseteq \pi(S) \supseteq \pi(Z) = Y$ which generates $G/K$. Therefore, $\pi|_M$ is a surjection, which means that $MK = G$. Furthermore, $M \supseteq S \supseteq X$ which generates $H$ so $K \leq_c H \leq_c M$ and thus $M = MK = G$ as required. 

\end{proof}

\subsubsection{Faithful action}

We establish \thmref{SecondRes} in a sequence of claims, the most important of which is the following theorem that makes crucial use of the assumption on linear growth of generating sets for subgroups (i.e. positive rank gradient). For the proof, recall that if $U$ is an open subgroup of a finitely generated profinite group $\Gamma$, then $d(U) \leq d(\Gamma)[\Gamma : U]$ as can be seen from \cite[Corollary 3.6.3]{RZ}.

\begin{theorem} \label{FaithActThm}

Let $p$ be a prime number, let $\Gamma$ be a pro-$p$ group with a hereditarily linearly increasing rank, and let $H \leq_c \Gamma$ be a finitely generated subgroup of infinite index. Then the action of $\Gamma$ on $\Gamma/H$ is faithful.

\end{theorem}

\begin{proof}

Let $K \lhd_c \Gamma$ be the kernel of the action, and note that $K \leq_c H$. Towards a contradiction, suppose that $K \neq \{1\}$. By \corref{FreqFamCor}, we can find some $M \leq_c \Gamma$ maximal among those having $H$ as an open subgroup. Since $[M : H] < \infty$ and $[\Gamma : H] = \infty$ by assumption, we can pick some $x \in \Gamma \setminus M$. Set $N \defeq \langle M \cup \{x\} \rangle$, so that $[N : H] = \infty$ by the maximality of $M$. We see that 
\begin{equation} \label{NH1Eq}
d(N) \leq d(M) + 1 \stackrel{\ref{UpWProp}}{\leq} d(H) + 1
\end{equation}
so $N$ is finitely generated, and from the fact that $K$ is a nontrivial nonopen subgroup of $N$, we infer that $d(N) > 1$ (if $d(N) \leq 1$ then all nontrivial subgroups of the pro-$p$ group $N$ are open).

By \eqref{DefLIREq}, there exists an $\epsilon > 0$ such that for all $V \leq_o N$ we have

\begin{equation} \label{VrakIn0Eq}
d(V) \geq \epsilon(d(N)-1)[N : V] = \delta[N : V]
\end{equation}
where $\delta \defeq \epsilon(d(N)-1) > 0$.  By \propref{TorFreeProp}, $K$ is infinite, and this fact (along with the fact that $[N : H] = \infty$) is seen in the finite quotients of $N$. For instance, there exists some $U \lhd_o N$ such that
\begin{equation} \label{FinImageIn0Eq}
[N : UH], [UK : U] > \frac{2(d(H)+1)}{\delta}.
\end{equation}
We find that
\begin{equation}
\begin{split}
d(U) &\stackrel{\ref{GenCountLem}}{\leq} d(U / U \cap K) + d(U \cap H) \\
&= d(UK/K) + d(U \cap H) \\
&\leq d(UK) + d(H)[H : U \cap H] \\
&\leq d(N)[N : UK] + d(H)[UH : U] \\
&\stackrel{\ref{NH1Eq}}{\leq} (d(H)+1) \frac{[N : U]}{[UK : U]} + d(H) \frac{[N : U]}{[N : UH]} \\
&\leq \frac{2(d(H)+1)[N : U]}{\min\{[UK : U],[N : UH]\}} \stackrel{\ref{FinImageIn0Eq}}{<} \delta[N : U]
\end{split}
\end{equation}
which is a contradiction to \eqref{VrakIn0Eq}.

\end{proof}

For the next corollary, recall that given a subgroup $H$ of a profinite group $\Gamma$, the normalizer of $H$ in $\Gamma$ (the set of $\gamma \in \Gamma$ for which $\gamma H = H\gamma$) is denoted by $N_{\Gamma}(H)$. The normalizer is easily seen to be a subgroup.

\begin{corollary} \label{NormalizerCor}

Let $p$ be a prime number, let $\Gamma$ be a pro-$p$ group with a hereditarily linearly increasing rank, and let $H \neq \{1\}$ be a finitely generated subgroup of $\Gamma$. Then $[N_{\Gamma}(H) : H] < \infty$.

\end{corollary}

\begin{proof}

Suppose that $[N_{\Gamma}(H) : H] = \infty$. Since hereditarily linearly increasing rank is inherited by subgroups, we can apply \thmref{FaithActThm} to the action of $N_{\Gamma}(H)$ on its cosets modulo $H$, and get that
\begin{equation} \label{ContrEq}
H = \bigcap_{g \in N_{\Gamma}(H)} gHg^{-1} = \{1\}
\end{equation}
where the first equality stems from the normality of $H$ in $N_{\Gamma}(H)$, and the second one from the faithfulness of the action of $N_{\Gamma}(H)$ on $N_{\Gamma}(H)/H$. Clearly, \eqref{ContrEq} contradicts our assumption that $H \neq\{1\}$.

\end{proof}

\subsubsection{Roots and Commensurators}

Given a subgroup $H$ of a profinite group $\Gamma$, we define the family of 'finite extensions' of $H$ by $\mathcal{F} \defeq \{R \leq_c \Gamma \ | \ H \leq_o R\}.$ Following \cite{R}, we say that $M \in \mathcal{F}$ is the root of $H$ (and write $M = \sqrt{H}$) if $M$ is the greatest element in $\mathcal{F}$ with respect to inclusion. Note that $\mathcal{F}$ may fail to have a greatest element, so $H$ does not necessarily have a root.

\begin{theorem} \label{RootThm}

Let $p$ be a prime number, and let $\Gamma$ be a pro-$p$ group with a hereditarily linearly increasing rank. Then every finitely generated subgroup of $\Gamma$ has a root. 

\end{theorem}

\begin{proof}

Let $n \in \mathbb{N}$, and let $\mathcal{D}$ be the family of all $n$-generated subgroups of $\Gamma$ which do not have a root. Towards a contradiction, suppose that $\mathcal{D} \neq \emptyset$. By \thmref{FratThm}, there is a maximal $M \in \mathcal{D}$. By \corref{FreqFamCor}, there exists some $T$ maximal among the subgroups of $\Gamma$ having $M$ as an open subgroup. Since $M \in \mathcal{D}$, we know that $T$ is not a root of $M$, so there exists an $A \leq_c \Gamma$ not contained in $T$, such that $M \leq_o A$. As $M \leq T$ and $A \nleq T$, we conclude that $M \lneq A$. The maximality of $T$ implies that $T \nleq A$ so $M \lneq T$. By \propref{TorFreeProp}, $M \neq \{1\}$ since otherwise $T$ would have been a nontrivial finite subgroup of $\Gamma$.

Set $B \defeq T \cap A$ and suppose that there exists a root $C$ of $B$. Since $M \leq_o B$ we see that $B \leq_o A,T$ so $A,T \leq_c C$ by definition of the root. Since $A \nleq T$, we find that $T \lneq C$. On the other hand, $C = \sqrt{B}$ implies that $B \leq_o C$, and thus $M \leq_o C$ which is a contradiction to the maximality of $T$. We conclude that $B$ does not have a root, and by \propref{UpWProp}, $d(B) \leq d(M) \leq n$. Hence, $B \in \mathcal{D}$ and the maximality of $M$ implies that $B = M$.

Let $A_0, T_0 \gneq_o M$ be minimal subgroups of $A,T$ respectively, and set $J \defeq \langle A_0 \cup T_0 \rangle$. Evidently, $M$ is a maximal subgroup of the pro-$p$ groups $A_0,T_0$ so by \cite[Lemma 2.8.7 (a)]{RZ}, $M \lhd_o A_0,T_0$ which means that $A_0,T_0 \leq_c N_{\Gamma}(M)$. Consequently, $J \leq_c N_{\Gamma}(M)$, so from \corref{NormalizerCor} we get that $[J : M] < \infty$. By \propref{UpWProp}, $d(T_0) \leq d(M) \leq n$ so from the maximality of $M$ we infer that $T_0$ has a root $R$. Since $T_0 \leq_o T,J$ we have $T,J \leq_c \sqrt{T_0} = R$. The fact that $T \cap A = M$ implies that $A_0 \nleq T$, so $T \lneq R$ as $A_0 \leq_o J \leq_o R$. Since $R = \sqrt{T_0}$ we have $M \leq_o T_0 \leq_o R$ so $R$ contains $M$ as an open subgroup, thus contradicting the maximality of $T$.

\end{proof}

The following assertion holds whenever roots exist.

\begin{proposition} \label{EqRootProp} 

Let $p$ be a prime number, let $\Gamma$ be a pro-$p$ group with a hereditarily linearly increasing rank, let $H$ be a finitely generated subgroup of $\Gamma$, and let $K$ be an open subgroup of $H$. Then $\sqrt{H} = \sqrt{K}$.

\end{proposition}

\begin{proof}

Since $K \leq_o H$ we have $H \leq_o \sqrt{K}$ so $\sqrt{K} \leq_o \sqrt{H}$. On the other hand, $K \leq_o H \leq_o \sqrt{H}$ so $K \leq_o \sqrt{H}$ and thus $\sqrt{H} \leq_o \sqrt{K}$.

\end{proof}

More generally, we have the following.

\begin{corollary} \label{EqRootCor} 

Let $p$ be a prime number, let $\Gamma$ be a pro-$p$ group with a hereditarily linearly increasing rank, and let $H,K$ be finitely generated subgroups of $\Gamma$. Then $H$ and $K$ are commensurable if and only if $\sqrt{H} = \sqrt{K}$.

\end{corollary}

\begin{proof}

Suppose first that $\sqrt{H} = \sqrt{K}$ and denote by $L$ this common root. We have
\begin{equation}
[H : H \cap K], [K : H \cap K] \leq [L : H \cap K] \leq [L : H][L : K]
\end{equation}
where the right hand side is finite since $L = \sqrt{H} = \sqrt{K}$ so we have established commensurability. Conversely, if $H$ and $K$ are commensurable then 
\begin{equation}
\sqrt{H} \stackrel{\ref{EqRootProp}}{=} \sqrt{H \cap K} \stackrel{\ref{EqRootProp}}{=} \sqrt{K}.
\end{equation}

\end{proof}

Recall that given a subgroup $H$ of a profinite group $\Gamma$, the commensurator of $H$ in $\Gamma$ (the set of $\gamma \in \Gamma$ for which $H$ and $\gamma H \gamma^{-1}$ are commensurable) is denoted by $\mathrm{Comm}_{\Gamma}(H)$. Commensurability is an equivalence relation on subgroups, and the commensurator is an abstract subgroup of $\Gamma$. 

\begin{corollary} \label{FinalProfCor}

Let $p$ be a prime number, let $\Gamma$ be a pro-$p$ group with a hereditarily linearly increasing rank, and let $H \neq\{1\}$ be a finitely generated subgroup of $\Gamma$. Then $\mathrm{Comm}_{\Gamma}(H) = \sqrt{H}$.

\end{corollary}

\begin{proof}

Let $g \in \sqrt{H}$. We have 
\begin{equation}
\sqrt{gHg^{-1}} = g\sqrt{H}g^{-1} = \sqrt{H}
\end{equation}
so by \corref{EqRootCor}, $H$ and $gHg^{-1}$ are commensurable, which means that $g \in \mathrm{Comm}_{\Gamma}(H)$. Now, let $x \in \mathrm{Comm}_{\Gamma}(H)$. Thus, $H$ and $xHx^{-1}$ are commensurable. By \corref{EqRootCor},
\begin{equation}
\sqrt{H} = \sqrt{xHx^{-1}} = x\sqrt{H}x^{-1} 
\end{equation} 
so $x \in N_{\Gamma}(\sqrt{H})$. We have thus shown that $\mathrm{Comm}_{\Gamma}(H) \leq N_{\Gamma}(\sqrt{H})$. Since $H \leq_o \sqrt{H}$ \propref{UpWProp} tells us that $\sqrt{H}$ is finitely generated, so by \corref{NormalizerCor}, $[N_{\Gamma}(\sqrt{H}) : \sqrt{H}] < \infty$ which implies that $N_{\Gamma}(\sqrt{H}) = \sqrt{H}$ in view of the maximality of the root $\sqrt{H}$. Hence, $\mathrm{Comm}_{\Gamma}(H) \leq \sqrt{H}$ and we have an equality.

\end{proof}

Observe that \thmref{SecondRes} follows at once from \corref{FinalProfCor} and \thmref{FaithActThm}.

\section{Abstract Groups}

For a group $G$, we denote by $\widehat{G}$ the profinite completion of $G$.

\begin{corollary} \label{ProfRankBoundCor}

Let $G$ be a group, let $n \in \mathbb{N}$, let $\{A_i\}_{i \in I}$ be a directed family of $n$-generated subgroups, and set $A \defeq \langle A_i \rangle_{i \in I}$. Then $d(\widehat{A}) \leq n$.

\end{corollary}

\begin{proof}

Apply \corref{BasicProfCor} to $\widehat{A}$, and the closures $\{\overline{A_i}\}_{i \in I}$ in $\widehat{A}$.

\end{proof}

We are now up to proving \thmref{ThirdRes}.

\begin{corollary} \label{AbstFratCor}

Let $G$ be a group for which every subgroup $H$ whose profinite completion is finitely generated, is itself finitely generated. Let $n \in \mathbb{N}$, and let $\mathcal{F} \neq \emptyset$ be a family of $n$-generated subgroups of $G$. Then $\mathcal{F}$ contains a maximal element.

\end{corollary}

\begin{proof}

Let $\mathcal{C}$ be a chain in $\mathcal{F}$, and let $U$ be its union. Since $\mathcal{C}$ is a chain, \corref{ProfRankBoundCor} implies that $d(\widehat{U}) \leq n$. By our assumption on $G$, there exists a finite generating set $S \subseteq U$. Since $\mathcal{C}$ is a chain, we can find an $R \in \mathcal{C}$ which contains $S$, and thus all of $U$. Therefore, $U = R \in \mathcal{F}$ is an upper bound for $\mathcal{C}$. By Zorn's Lemma, $\mathcal{F}$ has a maximal element.

\end{proof}

Let us now briefly explain why \thmref{SecondRes} applies to limit groups.

\begin{proposition} \label{ProfCompLimProp}

Let $L$ be a limit group, and let $H \leq L$ be a subgroup with a finitely generated profinite completion. Then $H$ is finitely generated.

\end{proposition}

\begin{proof}

Suppose that $H$ is not finitely generated. By \cite[Theorem 3.2]{S}, $L$ decomposes as a graph of groups $Y$ with cyclic edge groups. This induces a decomposition of $H$ as a graph of groups $X$ which must be infinite since $H$ is not finitely generated. If $X$ has an infinite first Betti number, then $H$ surjects onto a free group of infinite rank, so in particular, its profinite completion is not finitely generated.

We may thus assume that the first Betti number of $X$ is finite, which implies that
\begin{equation}
X = C \cup T_1 \cup \dots \cup T_n
\end{equation}
where $C$ is compact, $n \in \mathbb{N}$, and $T_1, \dots, T_n$ are infinite trees with a unique leaf each. Hence, in order to show that the profinite completion of $H$ is not finitely generated, it is sufficient to show this for the fundamental group of $T_1$. For that, recall that every abelian subgroup of a limit group is finitely generated free abelian so infinitely many vertex groups in $T_1$ are not cyclic. Since edge groups are cyclic, by collapsing edges we can assure that all vertex groups in $T_1$ are not cyclic. It follows from fully residual freeness that every vertex group surjects onto $\mathbb{Z}^2$, so the fundamental group of $T_1$ surjects onto the fundamental group of a tree of groups in which every vertex group is $\mathbb{Z}^2$ and every edge group is either $\{1\}$ or $\mathbb{Z}$. The latter group surjects onto an infinite direct sum of $\mathbb{Z}/2\mathbb{Z}$ so its profinite completion is not finitely generated.

\end{proof}

The next corollary follows \cite[Corollary]{Hig}.

\begin{corollary} \label{AutCor}

Let $G$ be a limit group, let $\alpha$ be an automorphism of $G$, and let $H \leq G$ be a finitely generated subgroup which is mapped by $\alpha$ into itself. Then $\alpha(H) = H$.

\end{corollary}

\begin{proof}

Set 
\begin{equation} \label{defAutFam}
\mathcal{F} \defeq \{K \leq G \ | \ d(K) \leq d(H), \ \alpha(K) \lneq K\}
\end{equation} 
and suppose that $H \in \mathcal{F}$. By \corref{AbstFratCor} and \propref{ProfCompLimProp}, there exists a maximal $M \in \mathcal{F}$. It is easy to verify that $\alpha^{-1}(M) \in \mathcal{F}$, so $\alpha^{-1}(M) = M$ by maximality. Hence, $\alpha(M) = M$ - a contradiction. 

\end{proof}

\section*{Acknowledgments}

I would like to sincerely thank Henry Wilton for answering many questions, and for his proof of \propref{ProfCompLimProp}. Special thanks go to Pavel Zalesskii for all the helpful discussions. This research was partially supported by a grant of the Israel Science Foundation with cooperation of UGC no. 40/14.

\item Author's address: Open Space - Room Number 2, Schreiber Building (Mathematics), Tel-Aviv University, Levanon Street, Tel-Aviv, Israel.

\item Author's email: markshus@mail.tau.ac.il

\item Author's website: \url{markshus.wix.com/math}

\end{document}